\documentclass{patmorin}
\usepackage{pat}
\usepackage{graphicx}
\usepackage{datetime}

\usepackage{hyperref}
\hypersetup{colorlinks=true,linkcolor=blue,citecolor=blue,urlcolor=blue}

\usepackage{marvosym}
\makeatletter
\renewcommand*{\@fnsymbol}[1]{\ensuremath{\ifcase#1 %
   \or \text{\Sun}
   \or \text{\Mercury}
   \or \text{\Venus}
   \or \text{\Earth}
   \or \text{\Mars}
   \or \text{\Jupiter}
   \or \text{\Saturn}
   \or \text{\Uranus}
   \or \text{\Neptune}
   \or \text{\Pluto}
  \else\@ctrerr\fi}}
\makeatother

\date{}

\DeclareMathOperator{\exponential}{exponential}

\DeclareMathOperator{\erlang}{Erlang}

\newcommand{\tcal}{\mathcal{T}}

\title{\MakeUppercase{Notes on Growing a Tree in a Graph}\thanks{%
This research was partly supported by NSERC and by NSF grants DMS1661063
and DMS1362785.
Some of this work was carried out while Vida Dujmovi\'c, Pat Morin, and
Bruce Reed were visiting IMPA in Rio de Janeiro. They thank the Institute
for its hospitality. }}
\author{
   Luc~Devroye,\thanks{School of Computer Science, McGill University, Montr\'eal, Canada}\,\,
   Vida~Dujmovi\'c,\thanks{Department of Computer Science and Electrical Engineering, University of Ottawa, Ottawa, Canada}\,\,
   Alan~Frieze,\thanks{Department of Mathematical Sciences, Carnegie Mellon University, Pittsburgh, USA. Research supported in part by NSF Grant DMS1362785}\,\,
   Abbas~Mehrabian,\thanks{Department of Computer Science, University of British Columbia, Vancouver, Canada. 
   This author was
       supported by an NSERC Postdoctoral Fellowship and a Simons-Berkeley
       Research Fellowship. Part of this work was done while he was
       visiting the Simons Institute for the Theory of Computing at UC
       Berkeley.}\,\,
   Pat~Morin,\thanks{School of Computer Science, Carleton University, Ottawa, Canada}\,\, and
   Bruce~Reed\thanks{CNRS, Projet COATI, I3S (CNRS and UNS) UMR7271 and INRIA, Sophia Antipolis, France. Instituto Nacional de Matem\'atica Pura e Aplicada (IMPA), Brasil. Visiting
Research Professor, ERATO Kawarabayashi Large Graph Project, Japan.}
}

\begin{document}
\maketitle

\begin{abstract}
We study the height of a spanning tree $T$ of a graph $G$ obtained by
starting with a single vertex of $G$ and repeatedly selecting, uniformly
at random, an edge of $G$ with exactly one endpoint in $T$ and adding
this edge to $T$.
\end{abstract}

\newpage 
\tableofcontents
\newpage

\pagenumbering{arabic}
\section{Introduction}
Let $s$ be a vertex of a simple connected graph $G$ on $n$ vertices.
We build a sequence $T_1,T_2,\dots,T_n$ of random subtrees of $G$ as follows.
The tree $T_1$ has a single vertex, $s$.
For each $1<i\leq n$, tree $T_i$ is obtained by choosing
a uniformly random edge of $G$ with exactly one endpoint in $T_{i-1}$,
and adding the edge to $T_{i-1}$.
Note that $T_n$ is a (not necessarily uniform) random spanning tree of $G$ rooted at $s$, which we denote by $\tcal(G,s)$.
In this paper we study the height (maximum length of a root-to-leaf path) of $\tcal(G,s)$ and give several bounds for it in terms of parameters of $G$.

In the special case when $G$ is the complete graph, 
each tree $T_i$ is obtained from $T_{i-1}$ by choosing a uniformly random node of $T_{i-1}$ and joining a new leaf to that node.
This is the well studied \emph{random recursive tree} process,
and 
Devroye \cite{devroye:branching} and Pittel \cite{pittel:note} have
shown that the height of $T_n=\tcal(K_n,s)$ is
$(e+o(1))\ln n$ with probability $1-o(1)$.


\subparagraph{Our results.}
Let $D=D(G)$ and $\Delta=\Delta(G)$ denote the diameter and maximum degree of $G$, respectively, and let us denote 
the height of a tree $T$ by $h(T)$.
An obvious lower bound for $h(\tcal(G,s))$ is $D/2$.
We prove the following bounds hold with probability $1-o_n(1)$ for any $n$-vertex graph $G$ and any $s\in V(G)$.
(The notation $o_k(1)$ denotes the set of functions $f:\R\to\R$ such
that $f(k)\to 0$ as $k\to\infty$.)


\begin{enumerate}
  \item In Theorem~\ref{thm:alantree-upper-bound} we show $h(\tcal(G,s))\in O(\Delta(D+\log n))$.
  For $D\in \Omega(\log \Delta)$ this is tight: in~\thmref{genericlowerbound} we show that for every $\Delta\ge  2$ and every $D \ge e^6\ln\Delta$, there exist $G$ and $s$ with 
   $h(\tcal(G,s))\in \Omega(\Delta(D+\log n))$.

  \item If $G$ is $d$-degenerate (that is, every subgraph of $G$ has a vertex of degree at most $d$),
  then in Theorem~\ref{thm:alantree-upper-bound-degenerate} we show  $h(\tcal(G,s))\in
    O(\sqrt{d\Delta}(D+\log n))$.  The class of $O(1)$-degenerate
    graphs is quite rich and includes every minor-closed graph family. 
  This upper bound is tight, even for planar graphs ($d=5$), graphs
  of thickness $t$ ($d=5t$), and graphs of treewidth $k$ ($d=k$).
  (The concepts of Euler genus, thickness, and treewidth
    are defined in \secref{degeneracy-discussion}).

For $D\in \Omega(\log \Delta)$ and planar graphs (which are 5-degenerate) this is tight:
  in~\thmref{planartight} we show 
  for any $\Delta>2$ and $D>10^6\ln \Delta$
  there exists a planar graph $G$ and vertex $s$ with
  $h(\tcal(G,s))\in
      \Omega(\sqrt{\Delta}(D+\log n))$.
  
Also, for $D\in \Omega(\log \Delta)$ and $d\leq\Delta$ this is tight:
in~\thmref{alantree-lower-bound-degenerate} we show 
for any $\Delta>1$, $D>10^6\ln \Delta$
and $d \leq \Delta$
there exist a $d$-degenerate graph $G$ and vertex $s$ with
$h(\tcal(G,s))\in
    \Omega(\sqrt{d\Delta}(D+\log n))$
    
  \item  If $G$ has Euler genus less than $C\sqrt{\Delta}D/\log\Delta$, then
  $h(\tcal(G,s))\in O(\sqrt{\Delta}(D+\log n))$ 
  (see \thmref{alantree-upper-bound-genus})  .
  For $D\in \Omega(\log \Delta)$ and zero Euler genus this is tight:
  in~\thmref{planartight} we show 
  for any $\Delta>2$ and $D>10^6\ln \Delta$
  there exist a planar graph $G$ and vertex $s$ with
  $h(\tcal(G,s))\in
      \Omega(\sqrt{\Delta}(D+\log n))$

  \item For any $d,k\geq1$, if $G$ is the $d$-dimensional grid of side-length $k$ (which has
  $n=(k+1)^d$ vertices), we have $h(\tcal(G,s))\in O(dk+d^{5/3}\ln(k+1))$.
  If $k=2$ or $k/\ln(k+1)=\Omega(d^{2/3})$, we have $h(\tcal(G,s))\in \Theta(D)=\Theta(dk)$
  (see \thmref{alantree-upper-bound-hypercube} and \corref{bestgrid}).

  \item If $G$ has edge-expansion factor\footnote{The
  edge expansion factor and related quantities are defined in
  \secref{edge-expanders}.} (i.e., Cheeger constant) $\Phi$, then
  $h(\tcal(G,s))\in O(\Phi^{-1}\Delta\log n)$ (see \thmref{alantree-upper-bound-cheeger}).  This implies, for example,
  that $h(\tcal(G,s))\in O(\log n)$ if $G$ is the complete graph or if $G$ is
  a random $\Delta$-regular graph (since a random $\Delta$-regular
  graph has $\Phi\in \Omega(\Delta)$, see~\cite{bollobas:isoperimetric}).
\end{enumerate}

Our main tool for proving upper bounds, \lemref{meta-theorem}, bounds
$h(\tcal(G,s))$ in terms of the first-passage percolation cover time and
the number of paths of  a given length starting at $s$.  To prove our
results using this tool, we prove several new bounds on first-passage
percolation cover times  as well as the number of simple paths in various
families of graphs, which are of independent interest.

\subparagraph{Our results on first-passage percolation cover time.}
Suppose independent exponential(1) random variables $\{\tau_e\}$ are assigned to edges of $G$.
Let $\Gamma(s,v)$ denote the set of all $(s,v)$-paths in the graph.
Then the \emph{first-passage percolation cover time} is defined as 
\[
	\tau(G,s) = \max_{v\in V(G)} \min_{\gamma\in\Gamma(s,v)} \sum_{e\in\gamma} \tau(e)
\]
In \lemref{percolation} we show a general upper bound of 
$O(\ln n + D)$ for $\tau(G,s)$. (This and the following results hold with probability $1-o_n(1)$.)

In the special case when $G$ is the $d$-dimensional grid with side length $k$ (and diameter $dk$), we prove the improved bound 
$\tau(G,s)=O(k)$.
The special case of $k=1$, namely the $d$-cube graph, was studied by Fill and Pemantle~\cite{fill.pemantle:percolation}, who showed $1.414\leq \tau(G,s)\leq 14.041$.
The upper bound was subsequently improved to
$1.694$
by Bollob\'as and Kohayakawa~\cite{bollobas.kohayakawa:on} 
and recently to $1.575$ by Martinsson~\cite{martinsson:unoriented}.


The remainder of the paper is organized as follows: 
\Secref{inequalities} presents some preliminaries
and useful facts about sums of independent  random variables,
In~\secref{sec:fpp_connection} we present the connection with first-passage percolation and prove a general upper bound.
\Secref{degeneracy}--\secref{grids} present our
upper bounds on $h(T)$.  
\Secref{lower-bounds} and
\secref{lowerdegenerate}
present families of
graphs with matching lower bounds.

We use the following notational conventions: $\log x$ denotes the binary
logarithm of $x$ and $\ln x$ denotes the natural logarithm of $x$.
Every graph, $G$, that we consider
is finite, simple, undirected and connected, and $n$ denotes its number of vertices.

\section{Preliminaries}
\label{sec:prelim}
\seclabel{inequalities}

Recall that an $\exponential(\lambda)$ random variable, $X$, has a
distribution defined by
\[
   \Pr\{X>x\} = e^{-\lambda x}  \enspace , \enspace x\ge 0 \enspace ,
\]
and mean $\E[X] = \int_0^{\infty}
\Pr\{X>x\}\,\mathrm{d}x = 1/\lambda$.  We make extensive use of the fact
that exponential random variables are \emph{memoryless}:
\[
    \Pr\{X > t+x\mid X>t\} = \frac{\Pr\{X> t+x\}}{\Pr\{X > t\}}
           = \frac{e^{-\lambda(t+x)}}{e^{-\lambda t}} = e^{-\lambda x} = \Pr\{X > x\} \enspace .
\]
We will also often take the minimum of $\delta$ independent $\exponential(\lambda)$ random variables and use the fact that this is distributed like an
$\exponential(\lambda \delta)$ random variable:
\[
   \Pr\{\min\{X_1,\ldots,X_\delta\} > x\} 
      = (\Pr\{X_1 > x\})^{\delta}
      = e^{-\delta\lambda x} \sim \exponential(\lambda \delta) \enspace .
\]
We will make use of two concentration inequalities for sums of exponential random
variables, both of which can be obtained using Chernoff's bounding
method (see, e.g., \cite[Theorem~5.1]{janson}).  If $Z_1,\ldots,Z_k$ are independent $\exponential(\lambda)$
random variables (so that they each have mean $\mu=1/\lambda$), then
for all $d>1$,
\begin{equation}
    \Pr\left\{\sum_{i=1}^k Z_i \le \mu k/d\right\} \le \exp(-k(\ln d -1 + 1/d)) \le \left(\frac{e}{d}\right)^k  \eqlabel{head-bound}
\end{equation}
and for all $t>1$, 
\begin{equation}
    \Pr\left\{\sum_{i=1}^k Z_i \ge \mu k t\right\} \le \exp (k-kt/2) \enspace . \eqlabel{tail-bound}
\end{equation}
The sum of $k$ independent $\exponential(\lambda)$
random variables is called an $\erlang(k,\lambda)$ random variable.


For positive integers $a$ and $b$, we define the random variable $Y_{a,b}$
as follows: Consider a tree in which the root has $a$ children,
 and each of the root's children have $b$ children.
Put an independent exponential(1) weight on each edge.  Then $Y_{a,b}$
is defined as the minimum weight of a path from the root to a leaf.
%
The following auxiliary lemma is proved in \appref{yab}.

\begin{lem}
\lemlabel{sqrtab}
Let $X_1,\dots,X_m$ be i.i.d.\ distributed as $Y_{a,b}$ for some $a,b$.
Then 
\[ \E[X_1] =O(1/a+1/\sqrt{ab}) \]
and moreover,
\[
\Pr\left\{\sum_{i=1}^m X_i \geq 3 m (64/a + 1024/\sqrt{ab}) \right\}
\leq
\exp(-m/9)\:.
\]
\end{lem}

\section{Connection with first-passage percolation and a generic upper bound}
\seclabel{sec:fpp_connection}

In this section, we establish the connection with first-passage percolation, and prove an upper bound for $\tau(G,s)$ in general graphs, which results in an upper bound for $h(\tcal(G,s))$.
This connection will be used in subsequent sections to provide tighter bounds for $h(\tcal(G,s))$ in several graph classes.


Recall the generation process for $\tcal(G,s)$:
we start with a tree containing only vertex $s$ initially;
in each round, we choose an edge uniformly at random
among edges with exactly one endpoint in the existing tree,
and add it to the existing tree.

We may view this as an \emph{infection process}:
at round 0 only vertex $s$ is infected.
In each round, suppose the set of infected vertices is $S$.
We choose a uniformly random edge between $S$ and its compliment, and let the disease spread along that edge, hence increasing the number of infected vertices by one.

Now consider the following 
continuous time view of this infection process, which is  known as Richardson's model~\cite{richardson_survey} or first-passage percolation~\cite{fpp_survey}.
At time $0$ we infect vertex $s$.
For each edge $uv$, whenever one of $u$ and $v$ gets infected,
we put an exponential(1) timer on edge $uv$.
When the timer rings, the disease spreads along that edge and both $u$ and $v$ get infected (it might be the case that both $u$ and $v$ are already infected by that time).
Suppose at some moment in this process, the subset $S$ of vertices are infected. Then, by memorylessness of the exponential distribution, the disease is equally like to spread along any of the edges existing between $S$ and its complement.
Therefore, the tree along which the disease spreads has the same distribution as $\tcal(G,s)$.

This viewpoint induces weights on the edges:
to each edge $e$ we  assign weight $\tau(e)$, which is the ringing time for the timer on this edge.
Note that the weights are i.i.d.\ exponential (1) random variables.
The weight of a path $P$, denoted $\tau(P)$, is simply the sum of weights of its edges.  The \emph{first-passage percolation hitting time} (or simply, the \emph{hitting time}) for $v$ is the weight of the lightest path from $s$ to $v$:
\[
    \tau(G,s,v) = \min_{\gamma\in\Gamma(s,v)} \tau(\gamma) \enspace .
\]
The \emph{first-passage percolation cover time} (or simply, the \emph{cover time}) is the first time that all vertices are infected, which can be written as
\[
	\tau(G,s) = \max_{v\in V(G)} \tau(G,s,v)  \enspace .
\]
Note that this is also the maximum \emph{weight} of a root-to-leaf path in the infection tree $\tcal(G,s)$, 
which we will use to bound the {height} of $\tcal(G,s)$,
the maximum \emph{length} of a path in (the unweighted version of) $\tcal(G,s)$ (in general, the longest path and the heaviest path may be different).

%
%

For a positive integer $L$ and a vertex $s$ of graph $G$, let $\Pi(G,s,L)$ denote the number of simple paths of length $L$ in $G$ that start from $s$.
We now prove a lemma that upper bounds $h(\tcal(G,s))$ in terms of $\tau(G,s)$ and $\Pi(G,s,L)$.

\begin{lem}\lemlabel{meta-theorem}
   Let $s\in V(G)$, $0\le p<1\leq a$, $c>0$, and
   $L=\ceil{ceaK}$ be such that $\Pr \{\tau(G,s)>K\}\leq p$
   and $\Pi(G,s,L) \leq a^L$. 
   Then  
    $h(\tcal(G,s)) \le L$ with probability at least $1-p-c^{-L}$.
\end{lem}

\begin{proof}
   Let $T = \tcal(G,s)$.
   If $h(T)> L$, then at least one of the following two events occurred:
   \begin{enumerate}
     \item $T$ contains a root-to-leaf path of weight greater than $K$.
     \item $G$ contains a path starting at $s$ of length $L$ whose weight 
        is less than $K$.
   \end{enumerate}
   By assumption, the probability of the first event is at most $p$.  
   The weight of a single path of length $L$ is the sum of $L$ $\exponential(1)$
   random variables so, by \eqref{head-bound} and the union bound over
   all $a^L$ paths, the probability of the second event is at most
   \[
       a^L \left(\frac{eK}{L} \right)^L \le c^{-L}
       \enspace . \qedhere
   \]
\end{proof}

In light of \lemref{meta-theorem}, we can upper bound $h(\tcal(G,s))$ if we have upper bounds on the cover time and on the number of paths of length $L$ originating at $s$.
An obvious upper bound for the latter is $\Delta^L$.
The following lemma gives a general upper bound for the former,
which results in a general upper bound for $h(\tcal(G,s))$.
In the following sections we obtain better bounds 
for these two quantities in special graph classes,
resulting in sharper bounds on $h(\tcal(G,s))$.


\begin{lem}\lemlabel{percolation}
For any $s\in V(G)$, we have
$\tau(G,s) \leq 4 \ln n + 2 D$ with probability at least
$1-1/n$.
\end{lem}

\begin{proof}
For each vertex $v\in V(G)$, we show the probability that it is not infected by time $4 \ln n + 2 D$ is at most $n^{-2}$, and then apply the union bound over all vertices.
Let $P$ be a shortest $(s,v)$-path in $G$.
Let $k\leq D$ denote the length of $P$, so $\tau(P)\sim \erlang(k,1)$.
Note that for any $t$,  $\tau(P)\leq t$ implies $v$ is infected by time $t$.
Thus, using \eqref{tail-bound}, the probability that $v$ is not infected by time $4 \ln n + 2 D$ is bounded by
\[
\Pr\{\tau(P) > 4 \ln n + 2 D\}
=
\Pr\{\erlang(k,1) > 4 \ln n + 2 D\}
\leq
\exp(k - 2\ln n - D) \leq n^{-2}\:.\qedhere
\]
\end{proof}

%

We immediately get a general upper bound for $h(\tcal(G,s))$.

\begin{thm}\thmlabel{alantree-upper-bound}
  Let $G$ be an $n$-vertex graph with diameter $D$ and maximum degree $\Delta>1$, and let $s$ be an arbitrary vertex.
Then, with probability at least $1-O(1/n)$ we have
\[
\frac D 2 \leq h(\tcal(G,s)) \leq 2e\Delta (4\ln n+2 D) \leq 
(4 e \Delta + 8 e\Delta  \ln \Delta )D + 16 e\Delta  \:.
\]
\end{thm}

Note that this gives an asymptotically tight bound of $h(\tcal(G,s))=\Theta(D)$ for graphs with bounded maximum degree.

\begin{proof}
The first inequality is trivial.
The second inequality is an application of \lemref{meta-theorem} with $a=\Delta$, $p=1/n$,
  $K=4\ln n+2D$ and $c=2$, using the bound of \lemref{percolation} for the cover time.
The last inequality follows from the crude bound $\Delta^D \geq n/3$, which holds for any $n$-vertex graph with maximum degree $\Delta$ and diameter $D$.
\end{proof}

%

\section{An upper bound in terms of graph degeneracy}
\seclabel{degeneracy}
Recall that a graph is \emph{$d$-degenerate} if each of its subgraphs has a vertex of  degree at most $d$.  The following lemma shows that, for large $L$, $d$-degenerate graphs have considerably less than $\Delta^L$ walks of length $L$.

\begin{lem}\lemlabel{few-walks}
   Let $G$ be an $n$-vertex $d$-degenerate graph with maximum degree
   $\Delta$.  Then the number of walks in $G$ of length $L$ is bounded by    $2n2^{L}(d\Delta)^{L/2}$.  
\end{lem}

\begin{proof}
   Enumerate the vertices of $G$ as $v_1,\ldots,v_n$ so that $v_i$ has at most
   $d$ edges in the subgraph induced by $v_i,\ldots,v_n$ (this ordering may be obtained by repeatedly removing a vertex of degree at most $d$).

   We give a way to encode the walks in a one-to-one way, and then bound the total number of possible generated codes.
   Let $W=v_{i_0},\ldots,v_{i_L}$
   be a walk of length $L$ in $G$ and let $k=k(W)$ denote the number
   of indices $\ell\in\{1,\ldots,L\}$ such that $i_{\ell-1} < i_{\ell}$.
   If $k\ge L/2$ then we say that $W$ is \emph{easy};
   note that at least one of $W$ and its reverse is easy, hence the total number of $L$-walks is at most twice the number of easy $L$-walks. We  encode an easy walk $W$ in the following way:
   \begin{enumerate}
     \item We first specify the starting vertex $v_{i_0}$.  There are $n$
       ways to do this.
     \item Next we specify whether $i_{\ell-1} < i_{\ell}$ for each
       $\ell\in\{1,\ldots,L\}$.  There are at most  $2^L$ ways to do this.
     \item Next, we specify each edge of $W$.  For each
       $\ell\in\{1,\ldots,L-1\}$, if $i_{\ell} < i_{\ell+1}$, then
       there are at most $d$ ways to do this, otherwise there are at
       most $\Delta$ ways to do this.
       Therefore, the total number of ways to specify all edges of the
       walk is at most
       \[   d^k\Delta^{L-k} \le (d\Delta)^{L/2}  \enspace ,\]
       since $d\le \Delta$ and $k\ge L/2$.
   \end{enumerate}
   Therefore, the number of easy $L$-walks is bounded by $n2^L(d\Delta)^{L/2}$, as required. \end{proof} 


\begin{thm}\thmlabel{alantree-upper-bound-degenerate}
  Let $G$ be an $n$-vertex $d$-degenerate graph with diameter $D$ and
  maximum degree $\Delta$, and let $s$ be an arbitrary vertex. 
  Then, with probability at least $1-O(1/n)$ we have
  $h(\tcal(G,s)) \leq 8e \sqrt{d\Delta}(2D+4\ln n)$.
\end{thm}

\begin{proof}
Let $c=2$, $K=4\ln n + 2D$, $p=1/n$, $a=4\sqrt{d\Delta}$, and $L=\lceil cea K\rceil > 8 \ln n$.
\lemref{percolation} guarantees
$\tau(G,s) \leq 4 \ln n + 2 D$ with probability at least
$1-1/n$,
and \lemref{few-walks} guarantees
$\Pi(G,s,L) \leq 2n2^{L}(d\Delta)^{L/2} \leq a^L$.
Applying \lemref{meta-theorem} completes the proof.
\end{proof}

\seclabel{degeneracy-discussion}

Note that \thmref{alantree-upper-bound-degenerate} actually implies
\thmref{alantree-upper-bound} up to constant factors, since all graphs of maximum degree $\Delta$
are $\Delta$-degenerate, so $\sqrt{d\Delta}\le \Delta$ in all cases.
However, \thmref{alantree-upper-bound-degenerate} provides sharper bounds for  many important graph classes:

\begin{itemize}
  \item Planar graphs are 5-degenerate. (This is a consequence of Euler's
    formula and the fact that planarity is preserved under taking subgraphs).
  
  \item The \emph{thickness} of a graph is the minimum number of planar
    graphs into which the edges of $G$ can be partitioned. Graphs of
    thickness $t$ are $5t$-degenerate.  (This follows from definitions
    and the $5$-degeneracy of each individual planar graph in the
    partition.)

  \item The \emph{Euler genus} of a graph is the minimum Euler genus of
    a surface on which the graph can be drawn without crossing edges.
    Graphs of Euler genus $g$ are $O(\sqrt{g})$-degenerate.\footnote{This
    follows from the facts in every $n$-vertex Euler-genus $g$ graph,
    $n\in \Omega(\sqrt{g})$ and there exists a vertex of degree at most
    $6+O(g/n)$. (See, e.g., \cite[Lemma~7 and Theorem~2]{wolle.koster.ea:note}.)}

  \item A \emph{tree decomposition} of a graph $G$ is a tree $T'$ whose
  vertex set $B$ is a collection of subsets of $V(G)$ called \emph{bags}
  with the following properties:
  \begin{enumerate}
    \item For each edge $vw$ of $G$, there is at least one bag $b\in B$
      with $\{v,w\}\subseteq B$.
    \item For each a vertex $v$ of $G$, the subgraph of $T'$ induced by
      the set of bags that contain $v$ is connected.
  \end{enumerate}
  The \emph{width} of a tree-decomposition is one less than the size
  of its largest bag.  The \emph{treewidth} of $G$ is the minimum
  width of any tree decomposition of $G$.
  Graphs of treewidth $k$ are $k$-degenerate. (This is a consequence
  of the fact that $k$-trees are edge-maximal graphs of treewidth $k$.)
\end{itemize} 

Therefore, \thmref{alantree-upper-bound-degenerate} implies that,
when the relevant parameter, $g$, $t$ or $k$, is bounded, $h(T)\in
O(\sqrt{\Delta}(D+\log n))$ with high probability.  

\section{An upper bound in terms of Euler genus}
Since graphs of Euler genus $g$ are $O(\sqrt{g})$-degenerate,
\thmref{alantree-upper-bound-degenerate} implies that if  $G$ has Euler genus
$g$, then $h(\tcal(G,s))\in O(g^{1/4}\Delta^{1/2}(D+\log n))$.  In this section we show that the
dependence on the genus $g$ can be eliminated when the diameter is large
compared to the genus.  We begin with a upper-bound on path counts that
is better (for graphs of small genus) than \lemref{few-walks}.

\begin{lem}\lemlabel{few-paths-genus}
   Let $G$ be a simple $n$-vertex graph of Euler genus $g$, diameter
   $D$, and maximum degree $\Delta\ge 6$. Then the number of simple paths
   in $G$ of length $L$ is at most $2n2^{L}6^{L/2-3g}\Delta^{L/2+3g}$.
\end{lem}

\begin{proof}
   The following proof makes use of some basic notions related to graphs
   on surfaces; see Mohar and Thomassen \cite{mohar.thomassen:graphs} for
   basic definitions and results.  Since $G$ has Euler genus $g$, it has
   a 2-cell embedding in a surface of Euler genus $g$.  Euler's formula
   then states that
   \begin{equation}
      m = n+f-2+g \enspace ,  \eqlabel{euler}
   \end{equation}
   where $n$ and $m$ are the numbers of vertices and edges of
   $G$ and $f$ is the number of faces in the embedding of $G$.  
   Every edge is on the boundary of at most 2 faces of the embedding
   and, since $G$ is simple, 
   every face is
   bounded by at least 3 edges.  Therefore, $f \le 2m/3$, so \eqref{euler}
   implies
   \[
       m\le 3n-6+3g \enspace .
   \]
   Therefore, the average degree of an $n$-vertex Euler genus $g$ graph
   is at most $6+(6g-12)/n$.  In particular, if $n \ge 6g$, then $g$
   has average degree less than 7, so $G$ contains a vertex of degree
   at most $6$.
  
   When we remove a vertex from $G$ we obtain a graph whose Euler genus is
   not more than that of $G$.  Therefore, by repeatedly removing a degree
   6 vertex, we can order the vertices of $G$ as $v_1,\ldots,v_n$ so that,
   for each $i\in\{1,\ldots,n-6g\}$, $v_i$ has at most 6 neighbours among
   $v_{i+1},\ldots,v_n$.  We call $v_{n-6g+1},\ldots,v_n$ \emph{annoying
   vertices} and edges between them are \emph{annoying edges}.

   Let $P=v_{i_0},\ldots,v_{i_L}$ be a path of length $L$ in $G$.
   For each $i\in\{1,\ldots,L\}$, the edge $v_{i_{\ell-1}}v_{i_{\ell}}$
   is called \emph{bad} if it is annoying or if $i_{\ell-1}>i_{\ell}$; otherwise it is called good. 
   Let $k$ denote the number of good edges in $P$.
   Say $P$ is good if $k \geq L/2-3g$.
   Note that the number of annoying edges of $P$ is bounded by $6g-1$, 
hence at least one of $P$ and its reverse is good.
We bound the number of good $L$-paths; the total number of $L$-paths is at most twice this bound.
We encode a good $L$-path $P$ as follows:
   \begin{enumerate}
     \item We first specify the starting vertex $v_{i_0}$.  There are $n$
       ways to do this.
     \item Next we specify whether each edge of $P$ is good or bad.
       There are $2^L$ ways to do this.
     \item Next, we specify each edge of $P$.  For each good edge,
       there are at most 6 ways to do this. For each bad edge there are 
       at most $\Delta$ ways to do this.
      Therefore, the total number of ways to specify the edges of $P$ is at most
      \[   6^k\Delta^{L-k} \le 6^{L/2-3g}\Delta^{L/2+3g}  \enspace ,\]
      since $k\ge L/2-3g$ and $\Delta\ge 6$.
   \end{enumerate}
   Therefore, the  number of good $L$-paths 
    is at most $n2^L6^{L/2-3g}\Delta^{L/2+3g}$, as required.
%
%
\end{proof}

\begin{thm}\thmlabel{alantree-upper-bound-genus}
  Let $G$ be an $n$-vertex Euler-genus $g$ graph with diameter $D$,
  maximum degree $\Delta$ and let $s\in V(G)$ be an arbitrary vertex.
  If $g\ln\Delta \le 36\sqrt{\Delta}(D+\ln n)$ then,
  with probability at least $1-O(1/n)$,
  $h(\tcal(G,s))\le 107\sqrt{\Delta}(2D+4\ln n)$.
\end{thm}

\begin{proof}
The conclusion follows from \thmref{alantree-upper-bound} for $\Delta\leq6$, so we will assume $\Delta>6$.
Let $c=2$, $K=4\ln n + 2D$, $p=1/n$, $a=8\sqrt{6\Delta}$, and $L=\lceil cea K\rceil > 8 \ln n$. 
\lemref{percolation} guarantees
$\tau(G,s) \leq 4 \ln n + 2 D$ with probability at least
$1-1/n$,
and \lemref{few-walks} guarantees
\[
\Pi(G,s,L) \leq 2n \times 2^{L} \times (6\Delta)^{L/2} \times \Delta^{3g} \leq 
(2 \times 2 \times \sqrt{6\Delta})^L
\exp\left( 108 \sqrt{\Delta}(D+\ln n)\right)
\leq
(2 \times 2 \times \sqrt{6\Delta}\times 2)^L=a^L
\:.\]
Applying \lemref{meta-theorem} completes the proof.
\end{proof}
%

\section{An upper bound for edge expanders}
\seclabel{edge-expanders}

All of the preceding upper bounds on $h(T)$ have a (linear or rootish)
dependence on $\Delta$, the maximum degree of a vertex in $G$.
This seems somewhat counterintuitive, since high degree vertices in $G$
should produce high degree vertices in $T$ and therefore decrease $h(T)$.
In this section we show that indeed large edge expansion (also called isoperimetric number
or Cheeger constant) results in low-height trees.

For an $n$-vertex graph $G$ and a subset $A\subseteq V(G)$,
define $e(A)=|\{vw\in E(G): v\in A,\, w\not\in A\}|$, and for any
$k\in\{1,\ldots,n-1\}$, define
\[
    e_k(G) = \min\{e(A) : A\subseteq V(G),\, |A|=k \} \enspace .
\]
Observe that $e_k(G)$ is symmetric in the sense that
\(e_k(G) = e_{n-k}(G) \enspace \).
We define the \emph{edge expansion} of $G$ is
\[
    \Phi(G) = \min\left\{e_k(G)/k : k\in\{1,\ldots,\lfloor n/2\rfloor\}\right\}
\]
We will express the height of $T$ in terms of the \emph{total inverse
perimeter size} $\Psi$, which is closely related to the edge expansion:
\[
    \Psi(G) = \sum_{k=1}^{\lfloor n/2\rfloor} \frac{1}{e_k(G)} 
            \le \sum_{k=1}^{\lfloor n/2\rfloor} \frac{1}{k\Phi(G)}
            = \frac{\ln n +O(1)}{\Phi(G)} \enspace .
\]


\begin{thm}\thmlabel{alantree-upper-bound-cheeger}
  \thmlabel{alantree-upper-bound-edge-expander}
  Let $G$ be an $n$-vertex graph with with maximum degree $\Delta$,
  edge-expansion $\Phi$, total inverse perimeter size $\Psi$,
  and let $s$ be an arbitrary vertex.
  Then, with probability at least $1-\exp(-\Omega(\Psi\Delta))$ we have
  $h(\tcal(G,s)) \in O(\Psi\Delta)\subseteq O(\Phi^{-1}\Delta\log n)$.
\end{thm}

Before proving \thmref{alantree-upper-bound-edge-expander}, we
consider the example of the complete graph $G=K_n$.  In this graph,
the minimum degree is $n-1$, so all preceding theorems (at best)
imply an upper bound of $O(n)$ on $h(\tcal(K_n,s))$.  However, $e_k(K_n) = k(n-k)$,
so $\Phi(K_n) = \lceil n/2\rceil$, and $\Psi(K_n) = O(\log n/n)$.
Then \thmref{alantree-upper-bound-edge-expander} implies that $h(\tcal(K_n,s))\in
O(\log n)$ with high probability. This upper bound
is of the right order of magnitude, since it matches the (tight) results
of Devroye and Pittel for the height of the random recursive tree
\cite{devroye:branching,pittel:note}.


\begin{proof}
   Fix some path $P=(s=v_0),v_1,\ldots,v_L$ in $G$ and suppose that $P$
   appears as a path in $T$.  Then there are times $1\leq k_1<\cdots<k_L<n$
   such that for each $i\in\{1,\ldots,L\}$, $v_i$ joins $T$ when
   $T$ has size $k_i$.  For a fixed $P$ and fixed $1\le k_1<\ldots<k_L<n$,
   the probability that this happens is at most
   \[
       \prod_{i=1}^{L} \frac{1}{e_{k_i}(G)} \:,
   \]  
   and the probability that $P$ appears in $T$ (without fixing
   $k_1,\ldots,k_L$) is at most
   \begin{align*}
       \sum_{1\le k_1<\cdots<k_L< n}
        \left(
         \prod_{i=1}^{L} \frac{1}{e_{k_i}(G)}
        \right) 
       < 
       \frac{1}{L!}\left(\sum_{(k_1,\ldots,k_L)\in\{1,\ldots,n-1\}^L}
        \left(
         \prod_{i=1}^{L} \frac{1}{e_{k_i}(G)}
        \right)\right)  
       = \frac{1}{L!}\left(\sum_{k=1}^{n-1}\frac{1}{e_k(G)}\right)^L 
       \le \frac{(2\Psi)^L}{L!} 
   \end{align*}
   Finally, since $G$ contains at most $\Delta^L$ paths of length $L$,
   
   \[
        \Pr\{h(\tcal(G,s)) \ge L\} \le \Delta^L \times \frac{(2\Psi)^L}{L!}  \leq \left(\frac{2e\Psi\Delta}{L}\right)^L
         \le \left(\frac{1}{2}\right)^L \enspace ,
   \]
   for $L\ge 4e\Psi\Delta$.
\end{proof}

Observe that the last step in the proof of
\thmref{alantree-upper-bound-cheeger} is to use the union bound over all
paths of length $L$.  If we have a better upper-bound than $\Delta^L$ on
the number of such paths, then we obtain a better upper bound on $h(T)$.
\lemref{few-walks} gives a better upper bound for $d$-degenerate graphs, using which we immediately obtain the following corollary.

\begin{cor}
  Let $G$ be an $n$-vertex $d$-degenerate graph with diameter $D$
  and maximum degree $\Delta$, total inverse perimeter size $\Psi$,
  and let $s$ be an arbitrary vertex.
  Then, with probability at least $1-O(1/n)$,
  $h(\tcal(G,s))\in O(\Psi\sqrt{d\Delta}+\log n)
  \in O(\log n (1+ \sqrt{d\Delta}/\Phi))$.
\end{cor}

\begin{proof}
As in the proof of \thmref{alantree-upper-bound-cheeger},
and using the upper bound 
$2n2^{L}(d\Delta)^{L/2}$ for the number of paths of length $L$, given by \lemref{few-walks}, we have
   \[
        \Pr\{h(\tcal(G,s)) \ge L\} \le  2n2^{L}(d\Delta)^{L/2}\times
        \frac{(2\Psi)^L}{L!} \leq
        2n\left( 4e\Psi\sqrt{d\Delta}/L \right)^L
        \leq
        \left( 8e\Psi\sqrt{d\Delta}n^{1/L}/L \right)^L\:,
    \]
    which is smaller than $1/n$ for $L\geq 8e^3 \Psi \sqrt{d\Delta}+\ln n$, as required.
\end{proof}

\section{Upper bounds for high dimensional grids and hypercubes}
\seclabel{grids}
The \emph{$d$-cube} is the graph having vertex set $\{0,1\}^d$ in
which two vertices are adjacent if and only if they differ in exactly
one coordinate.  Every vertex in the $d$-cube has degree $d$ and the
$d$-cube has diameter $d$.  The $d$-cube is an interesting example
in which the path count is high, but this is counteracted by a low
first-passage percolation time.

\begin{thm}\thmlabel{alantree-upper-bound-hypercube}
  Let $n=2^d$, let $G$ be the $d$-cube and let $s\in V(G)$ be arbitrary. Then,
  with probability at least $1-o_n(1)$, $h(\tcal(G,s))\in \Theta(d)$.
\end{thm}

\begin{proof}
  Fill and Pemantle \cite{fill.pemantle:percolation} showed that the
  first-passage percolation cover time for the $d$-cube is at most 14.05
  with probability $1-o_n(1)$.  Every vertex of the hypercube has degree
  $d$, so the number of paths of length $L$ starting at $s$ is less
  than $d^L$.  The result then follows by applying \lemref{meta-theorem}
  with $p=o_n(1)$, $c=2$, $K=14.05$, and $a=d$.
\end{proof}

A natural generalization of the $d$-cube is the \emph{$(d,k)$-grid}, which
has vertex set $\{0,\ldots,k\}^d$ and has an edge between two vertices
if and only if the (Euclidean or $\ell_1$) distance between them is 1.
The $(d,k)$-grid has diameter $D=dk$ and maximum degree $\Delta=2d$.

Note that in the case $k=1$, the $(d,1)$-grid is the $d$-cube, for which
\thmref{alantree-upper-bound-hypercube} gives the optimal bound and this
bound can be extended to $k\in O(1)$.   \thmref{alantree-upper-bound}
gives an upper bound of $O(d^2k)$ on $h(\tcal(G,s))$, which is optimal
for $d\in O(1)$.  The rest of this section is devoted to proving the
following result on the first-passage-percolation cover time of the
$(d,k)$-grid, which gives an optimal bound on the height of $\tcal(G,s)$
for all values of $k$ and $d$.

\begin{thm}\thmlabel{bestgrid-universal}
Let $G$ be the $(d,k)$-grid and $n=(k+1)^d$.  Then, for any vertex $s\in
V(G)$, we have that $\tau(G,s)= O(k)$ with probability $1-o_n(1)$.
\end{thm}

Before jumping into the proof, we note that applying \lemref{meta-theorem}
gives the following corollary of \thmref{bestgrid-universal}.

\begin{cor}\corlabel{bestgrid}
Let $G$ be the $(d,k)$-grid and $n=(k+1)^d$.  For any vertex $s\in V(G)$
we have that with probability $1-o_n(1)$, $ h(\tcal(G,s)) = \Theta(dk)$.
\end{cor}

To prove \thmref{bestgrid-universal}, we will make use of a concentration
result about the first-passage percolation time on the $d$-cube.

\begin{lem}\lemlabel{cube-diametrical}
   Let $Q$ be the $d$-cube, let $s\in V(Q)$ be arbitrary, and
   let $\bar{s}\in V(Q)$ be the unique vertex at distance $d$ from $s$.  
   Then there exist universal constants
   $c>0$ and $x_0>0$ such that,
   \[ 
     \Pr\{\tau(Q,s,\bar{s}) > x\} \le e^{-cxd} \enspace ,
   \]
   for all $x\ge x_0$.
\end{lem}

\begin{proof}
   We assume that $d$ is greater than some sufficiently large constant,
   $d_0$.  Otherwise the result follows trivially from the union bound:
   With probability at least $1-d2^{d-1} e^{-x/d}$, every edge of the
   $d$-cube has weight at most $x/d$.  For $d \le d_0$, this satisifies
   the statement of the lemma with $x_0=3d_0\ln d_0$ and $c=1/(3d_0^2)$.

   We will prove the result for all $x\le d^2$.  Proving it for $x$
   in this range is sufficient, by a standard bootstrapping argument:
   For $x > d^2$, let $r=2^{\lceil\log(x/d^2)\rceil}$, so that 
   \[ x_0 \le x/r \le d^2 \]
   where the first inequality holds provided that $d^2\ge 2x_0$.
   Consider a modified version of Richardson's infection model,
   which has the same rules as the original process
   except that, for each $i\in\N$, if the process has
   not infected $\bar{s}$ by time $ix/r$, then we restart the process
   from the beginning.  Clearly the time to infect $\bar{s}$ in this modified
   process dominates the time to infect $\bar{s}$ in the original process, so
   \[  \Pr\{\tau(Q,s,\bar{s}) > x\} \le \Pr\{\tau(G,s,\bar{s}) > x/r\}^{r} 
      \le \left(e^{-cxd/r}\right)^{r} = e^{-cxd} \enspace .
   \]

   Thus, it suffices to prove the lemma for $x_0\le x\le d^2$.
   For each $i\in\{0,1,\ldots,d\}$, let $L_i$ denote the subset
   of $\binom{d}{i}$ vertices whose distance to $s$ is $i$.  Balister \etal\
   \cite[Lemma~4]{balister.bollobas.ea:first-passage} show that there are
   constants $\alpha,\gamma>0$ such that, if we sample each edge of $Q$
   independently with probability $\alpha/d$ then, with probability at
   least $1-e^{-\gamma d^2}$, there is a path of length $d-4$ consisting
   entirely of sampled edges and having one endpoint in $L_2$ and one
   endpoint in $L_{d-2}$.

   In our setting, where edge weights are independent $\exponential(1)$,
   if we only consider edges of weight at most $\ln(d/(d-\alpha))$,
   then we obtain a sample in which each edge is independently
   sampled with probability $\alpha/d$.  Therefore, with probability
   at least $1-e^{-\gamma d^2}$, there is a path of weight at most
   $d\ln(d/(d-\alpha))=O(1)$ joining a vertex $u$ in $L_2$ to a vertex
   $w$ in $L_{d-2}$.

   Now, there are $d$ edge-disjoint paths of length $2$ joining $s$
   to $u$.  Consider one such path, $P$.  If $P$ has weight greater than
   $x/3$ then at least one of $P$'s edges has weight greater than $x/6$,
   which occurs with probability at most $2e^{-x/6}$.  Therefore,
   the probability that all $d$ paths have weight greater than $x/3$
   is at most $(2e^{-x/6})^d \le e^{-adx}$ for $0<a<1/6-\ln 2/x$.
   Similarly, with probability at least $1-e^{-adx}$, there is a path of
   length 2 and weight at most $x/3$ joining $w$ to $\bar{s}$.

   Therefore, 
   \[
      \Pr\{\tau(Q,s,\bar{s}) > 2x/3 + d\ln(d/(d-a))\} 
        \le \Pr\{\tau(G,s,\bar{s}) > x\}  
           \le 2e^{-axd}+e^{-\gamma d^2}
   \]
   provided that $x \ge\max\{3d\ln(d/(d-a)),\ln 2/(1/6-a)\}$. This holds, 
   for example, when $x\ge x_0 = 12\ln 2$, $a=1/12$, and $d\ge 1$.
\end{proof}

\lemref{cube-diametrical} extends from $\tau(Q,s,\bar{s})$ to
$\tau(Q,s,v)$ for any vertex $v\in V(Q)$ using the union bound and
the fact that $v$ is at distance at least $d/2$ from at least one
of $s$ or $\bar{s}$. (This argument is also used by Balister \etal\
\cite{balister.bollobas.ea:first-passage}.)

\begin{cor}\corlabel{cube-radius}
   Let $Q$ be the $d$-cube, and let $s,v\in V(Q)$ be arbitrary.
   Then there exist universal constants
   $c>0$ and $x_0>0$ such that,
   \[ 
     \Pr\{\tau(Q,s,v) > x\} \le e^{-cxd} \enspace ,
   \]
   for all $x\ge x_0$.
\end{cor}

%

\corref{cube-radius} says that the tail of $\tau(Q,s,v)$
is dominated by an $\exponential(cd)$ random variable.  The next lemma
shows that sums of independent copies of $\tau(Q,s,v)$ also behave (roughly)
like sums of independent $\exponential(cd)$ random variables.

\begin{lem}\lemlabel{cube-diameter-sum}
   Let $Q$ be a $d$-cube, let $s_1,\ldots,s_k\in V(Q)$ and
   $v_1,\ldots,v_k\in V(Q)$ be arbitrary, and let $Z_1,\ldots,Z_k$
   be i.i.d., random variables where $Z_i$ is distributed like
   $\tau(Q,s_i,v_i)$.  Then there exist universal constants $c>0$ and
   $x_0>0$ such that
   \[ \Pr\left\{\sum_{i=1}^k Z_i > ak\right\} \le e^{(x_0+1)k-acdk/2} \enspace ,\]
for all $a>0$.
\end{lem}

\begin{proof}
If $X$ is $\exponential(cd)$, then \corref{cube-radius} says that
\[
    \Pr\{\tau(Q,s_i,v_i) > x\} \le e^{-cxd} = \Pr\{X > x\}
\]
for all $x>x_0$.  This implies that
\[
    \Pr\{\tau(Q,s_i,v_i) > x\} \le \Pr\{X+x_0 > x\}
\]
for all $x>0$, i.e., $X+x_0$ stochastically dominates $\tau(Q,s_i,v_i)$.  Therefore, 
if $X_1,\ldots,X_k$ are independent $\exponential(1)$, then 
\[
    \Pr\left\{\sum_{i=1}^k Z_i > ak\right\} \le 
    \Pr\left\{\sum_{i=1}^k X_i > (a-x_0)k \right\} \le
    e^{(x_0+1)k-acdk/2} \enspace ,
\]
where the second inequality is an application of \eqref{tail-bound}.
\end{proof}

We can now finish the proof of \thmref{bestgrid-universal}.

\begin{proof}[Proof of \thmref{bestgrid-universal}]
The idea of this proof is that, for any vertex $v$, there is a path
from $s$ to $v$ that visits at most $k$ $d$-cubes. Therefore, there is
a path from $s$ to $v$ whose length can be expressed as a sum like
that considered in \lemref{cube-diameter-sum}.

For each vertex $u=(u_1,\ldots,u_d)$ with $u_i\in\{0,\ldots,k-1\}$
for each $i\in\{1,\ldots,d\}$, define the subgraph
$Q_u=G[V_u]$ of $G$ induced by the vertex set
\[
    V_u = \{(u_1+x_1,\ldots,u_d+x_d) :\text{$x_i\in\{0,1\}$ for each $i\in\{1,\ldots,d\}$}\} \enspace .
\]
Each $Q_u$ is a $(d,1)$-grid, i.e, a $d$-cube.  
For any vertex $v\in G$, there is a sequence $v_1,v_2,\ldots,v_{k'}$
of vertices in $G$ with $k'\le k$ such that
\begin{enumerate}
  \item $s\in Q_{v_1}$ and $v\in Q_{v_{k'}}$; 
  \item for each $i\in\{1,\ldots,k'-1\}$, $Q_{v_{i}}$ and $Q_{v_{i+1}}$ have at least one
vertex in common.
  \item for each $i\in\{1,\ldots,k'\}$ and each $j\in\{1,\ldots,k'\}\setminus\{i-1,i,i+1\}$, $Q_{v_i}$ and $Q_{v_j}$ have no edges (or vertices) in common.
\end{enumerate}
The sequence $v_1,\ldots,v_{k'}$ can be found with a greedy algorithm:
Define $v_0'=s$.  Now, if $v_{i-1}'$ and $v$ differ in $r$ coordinates,
then there is some vertex $v_i$ such that $Q_{v_i}$ contains $v_{i-1}'$
as well as some vertex $v_i'$ whose distance to $v$ is $r$ less than
the distance from $v_{i-1}'$ to $v$. It is straightforward to verify
that the resulting sequence of vertices $v_1,\ldots,v_{k'}$ satisfies
the three properties described above.

For each $i\in\{1,\ldots,k'\}$,
let $s_i=v_{i-1}'$ and $x_i=v_i'$.  Now, observe that
\[
   \tau(G,s,v) \le \sum_{i=1}^{k'} \tau(Q_{v_i},s_i,x_i)  \enspace .
\]
Point 3, above, ensures that the random variables
$\tau(Q_{v_1},s_1,x_1),\ldots,\tau(Q_{v_{k'}},s_{k'},x_{k'})$ can be
partitioned into two sets of size $\lfloor k'/2\rfloor$ and $\lceil
k'/2\rceil$ where the variables within each set are independent.

By \lemref{cube-diameter-sum} we now have
\begin{align*}
   \Pr\{ \tau(G,s,v) > ak\} 
       & \le \Pr\left\{\sum_{i=1}^{\lceil k'/2\rceil} \tau(Q_{v_{2i-1}},s_{2i-1},x_{2i-1}) > ak\right\} 
          + \Pr\left\{\sum_{i=1}^{\lfloor k'/2\rfloor} \tau(Q_{v_{2i}},s_{2i},x_{2i})>ak\right\} \\
          & \le 2 e^{(x_0+1)k - acdk/2} \enspace .
\end{align*}
Applying the union bound over all $(k+1)^d$ choices of $v$ completes the proof:
\begin{align*}
   \Pr\{\tau(G,s) > ak\} 
     & \le \sum_{v\in V(G)} \Pr\{\tau(G,s,v) > ak\}  \\
     & \le 2(k+1)^d e^{(x_0+1)k - acdk/2} \\
     & = e^{d\ln(2(k+1))+(x_0+1)k - acdk/2} \\
     & = o_n(1) \enspace ,
\end{align*}
for
\[
   a > \frac{2(d\ln(2(k+1)+(x_0+1)k)}{cdk} \enspace . \qedhere
\]
\end{proof}

\section{Lower Bounds for General Graphs}
\seclabel{lower-bounds}

Next, we describe a series of lower bound
constructions that match the upper bounds obtained in
Theorems~\ref{thm:alantree-upper-bound}--\ref{thm:alantree-upper-bound-genus}.
In particular, these constructions show that the dependence on $\Delta$
in the upper bounds in the previous sections can not be asymptotically
reduced.

In this section we prove the following theorem.

\begin{thm}
\thmlabel{genericlowerbound}
There exists a positive constant $c$ such that 
for any given positive integers $1<\Delta, D$ satisfying
$D\geq 16e^3 \ln \Delta$,
there exists an $n$-vertex graph $G$ with 
maximum degree $\leq \Delta$,
diameter $\leq D$, and
a vertex $s$ satisfying
$\Pr\{h(\tcal(G,s)) \geq c(\Delta \ln n + \Delta D)\}\geq 1-o_n(1)$.
\end{thm}

The graph $G$ is obtained by gluing together two graphs $H$ and $I$.
The graph $H$ has large diameter and high connectivity.  The graph $I$
has low connectivity and small diameter.  By gluing them we obtain a graph
of low diameter (because of $I$) but for which the infection is more likely to spread via $H$ (because of its high connectivity), and hence will have a large height.  We begin by
defining and studying $H$ and $I$ individually.

\subsection{The Ladder Graph $H$}

Let $L,\delta,a$ be positive integers.
The graph $H$ is shown in~\figref{h}.  The vertices of $H$ are partitioned into $L$ groups
$V_1,\ldots,V_L$, each of size $\delta$. The edge set of $H$ is
\[
   E(H) = \bigcup_{i=1}^{L-1} \{vw : v\in V_{i},\, w\in V_{i+1}\} \enspace .
\]
\begin{figure}
  \begin{center}
    \includegraphics{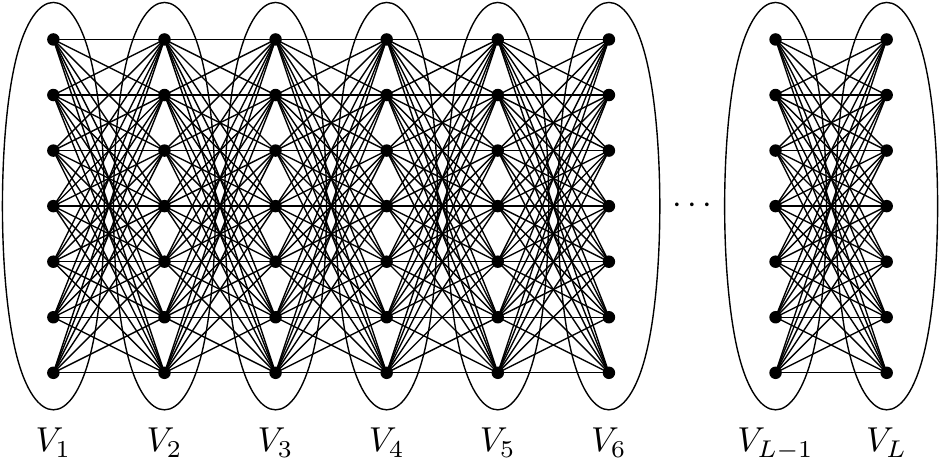} 
  \end{center}
  \caption{The graph $H$.}
  \figlabel{h}
\end{figure}

First we show that the infection spreads rather quickly in $H$, namely we prove upper bounds for $\tau(H,v,w)$.

\begin{lem}\lemlabel{lightning}
Let $a>e^2$. Then for any $1 \leq i < j \leq n$ and any $v\in V_i$, $w\in V_j$ we have
  \[
    \Pr\{\tau(H,v,w) > 2aL/(e^2\delta)\} 
        \le \exp(L-aL/(2e^2)) + \exp(-aL/(e^2\delta)) \enspace . 
  \]
\end{lem}

\begin{proof}
   Consider the following greedy algorithm for finding a path from $v$
   to $w$: The path starts at $v$ (which is in $V_i$).  When the path
   has reached some vertex $x\in V_{k}$, for $k<j-1$, the algorithm
   extends the path by taking the minimum-weight edge joining $x$ to some
   vertex in $V_{k+1}$.  When the algorithm reaches some $x\in V_{j-1}$,
   it takes the edge $xw$.

   Let $m=j-i$.  Each of the first $m-1$ edges in the resulting path
   has a weight that is the minimum of $\delta$ $\exponential(1)$
   random variables.  Thus, the sum of weights of these edges is the sum of
   $m-1$ $\exponential(\delta)$ random variables, i.e.\ an $\erlang(m-1,\delta)$ random variable.
   The weight of the final edge is an $\exponential(1)$ random variable.
   Thus we find
\begin{align*}
\Pr\left\{\tau(H,v,w)> 2aL/(e^2\delta) \right\} & \leq 
\Pr\left\{\erlang(m-1,\delta)+\exponential(1) > 2aL/(e^2\delta) \right\}\\
& \leq 
\Pr\left\{\erlang(m-1,\delta) > aL/(e^2\delta) \right\}
+
\Pr\left\{\exponential(1) > aL/(e^2\delta) \right\} \\
& \leq 
\Pr\left\{\erlang(L,\delta) > aL/(e^2\delta) \right\}
+
\exp(-aL/(e^2\delta)) \\
& \le \exp(L-aL/(2e^2)) + \exp(-aL/(e^2\delta)) \:.
\end{align*}   
The first inequality follows from the discussion above.
The second inequality follows from the union bound.
The third inequality is because an $\erlang(L,\delta)$ random variable stochastically dominates an $\erlang(m-1,\delta)$ random variable,
and the definition of the exponential distribution.
The final equality follows from the tail bound \eqref{tail-bound}.
\end{proof}


\subsection{The Subdivided Tree $I$}

Next, we consider a tree $I$ that is obtained by starting with a perfect
binary tree\footnote{A perfect binary tree, sometimes called a complete binary tree, is a binary tree in which all vertices have 0 or 2 children, and all leaves have the same depth: \url{https://xlinux.nist.gov/dads/HTML/perfectBinaryTree.html}} having $L$ leaves and then subdividing each edge incident
to a leaf $\lceil aL/\delta\rceil-1$ times so that each leaf-incident
edge becomes a path of length $\lceil aL/\delta\rceil$.  Note that $I$
has height $\lceil aL/\delta\rceil+\log_2 L-1$ (we assume $L$ is a power of 2).


We next show that the infection spreads rather slowly in $I$, namely we prove lower bounds for $\tau(I,v,w)$.

\begin{lem}\lemlabel{slow-train}
For any distinct leaves $v$ and $w$ we have
   $\Pr\{\tau(I,v,w) \le 2aL/(e^2\delta)\} \le \exp(-2aL/\delta)$.
\end{lem}

\begin{proof}
  The path from $v$ to $w$ in $I$ contains at least $2\lceil
  aL/\delta\rceil$ edges.  Therefore, the weight of this path is
  lower-bounded by the sum of $2\lceil aL/\delta\rceil$ independent
  $\exponential(1)$ random variables.  The lemma then follows by applying
  \eqref{head-bound} to this sum.
\end{proof}

\subsection{Putting it Together}

The lower-bound graph $G$ is now constructed by taking a tree $I$ with
$L$ leaves and a graph $H$ with $L$ groups $V_1,\ldots,V_L$ each of size
$\delta=\lfloor(\Delta-1)/2\rfloor$.  Next, we consider the leaves of
$I$ in the order they are encountered in a depth first-traversal of $I$
and, for each $i\in\{1,\ldots,L\}$ we identify the $i$th leaf of $I$
with some vertex in $V_i$.  See \figref{g}.

\begin{figure}
  \begin{center}
    \includegraphics{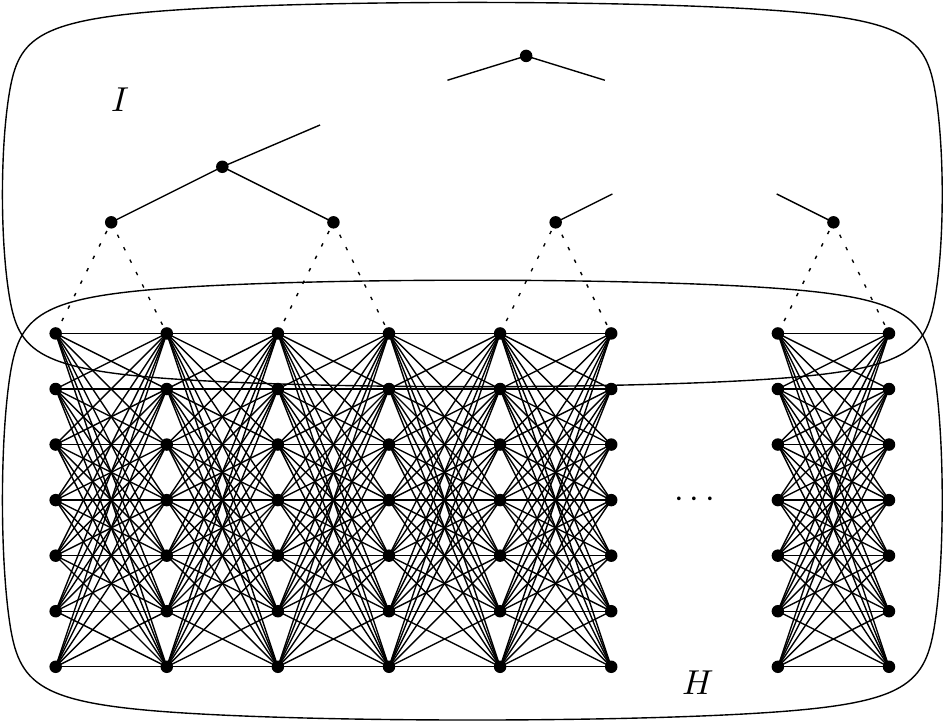}
  \end{center}
  \caption{The lower bound graph $G$. Dashed segments denote subdivided edges (paths of length $\lceil aL/\delta\rceil$).}
  \figlabel{g}
\end{figure}

\begin{lem}\lemlabel{lowerboundlemma}
For any vertex $s\in V_1$ in the graph $G$ described above, we have
\[
\Pr\{h(\tcal(G,s)) < L-1\}
\le L^2\left(\exp((1-a/2e^2)L) 
                + \exp(-aL/(e^2\delta)) 
                + \exp(-2aL/\delta)
              \right)
\]
\end{lem}

\begin{proof}
Recall that $\tcal(G,s)$ is the shortest-path tree rooted at $s$ for the first-passage percolation in $G$.
If this tree contains no edge of $I$,
its height is at least $L-1$.
If it does use some edge of $I$, then there must be two leaves $v$ and $w$ of $I$ such that $\tau(I,v,w)\le \tau(H,v,w)$. Since there are $\binom{L}{2}<L^2$ choices for the pair $\{v,w\}$,
using \lemref{lightning} and \lemref{slow-train}, we can bound the probability of this event by
  \begin{align*}
       & L^2\left(\Pr\{d_H(v,w) > 2aL/\delta\} + \Pr\{d_I(v,w) < 2aL/\delta\}\right)  \\
       & \le L^2\left(\exp((1-a/2e^2)L) 
                + \exp(-aL/(e^2\delta)) 
                + \exp(-2aL/\delta)
              \right)\:,
  \end{align*}
which proves the lemma.
\end{proof}

We  now have all the ingredients to prove the main theorem of this section,
\thmref{genericlowerbound}.

\begin{proof}[Proof of~\thmref{genericlowerbound}]
Let $a=4e^2$, $\delta = (\Delta-1)/2$, and let $L$ be the largest power of 2 that is not larger than $D \Delta / 8a$.
Let $G$ be the graph described above.
The maximum degree of $G$ is $2\delta+1 = \Delta$, and the diameter of $G$ is bounded by
\[
2 (aL/\delta + \log_2 L) \leq 
2 (a \times (D \Delta / 8a) / (\Delta/2) + 
\log_2(D \Delta / 8a))\leq
D \:,
\]
and its number of vertices is 
\[
n = L\delta + (2L-1) + L(aL/\delta-1)<L(\delta+1+aL/\delta)\:.
\]
We have
\[
L \geq D \Delta /4a
= \Omega( D\Delta + \Delta \ln L + \Delta \ln (\delta+1+aL/\delta)) = \Omega(\Delta \ln n + \Delta D)\:.
\]
By~\lemref{lowerboundlemma}, there exists a vertex $s$ such that
\begin{align*}
&\Pr\{h(\tcal(G,s)\geq\Omega(\Delta \ln n + \Delta D))\}
 \geq
\Pr\{h(\tcal(G,s)\geq L-1)\} \\
&\geq 1 - L^2\left(\exp((1-a/2e^2)L) 
                - \exp(-aL/(e^2\delta)) 
                - \exp(-2aL/\delta)
              \right)\\
&= 1 -\left(\exp(-L+2\ln L) 
                - \exp(-8L/\Delta+2\ln L) 
                - \exp(-16e^2L/\Delta+2\ln L)
              \right)=1-o_L(1)=1-o_n(1)\:,
\end{align*}
completing the proof.
\end{proof}

\section{Lower Bounds for Degenerate Graphs}
\seclabel{lowerdegenerate}

\thmref{genericlowerbound} shows that \thmref{alantree-upper-bound}
cannot be strengthened without knowing more about $G$ than its number of vertices, maximum
degree, and diameter.  \thmref{alantree-upper-bound-degenerate} provides a stronger
upper bound under the assumption that $G$ is $d$-degenerate.  
In this section we
show that \thmref{alantree-upper-bound-degenerate} is also tight, even
when restricted to  very special subclasses of $d$-degenerate graphs.

First we show that the bound given by
\thmref{alantree-upper-bound-degenerate} for $O(1)$-degenerate graphs
is tight, even when we restrict our attention to planar graphs, which
are 5-degenerate.  Since planar graphs have genus 0, this lower bound
also shows that \thmref{alantree-upper-bound-genus}, which applies to
bounded genus graphs, is  tight.

\begin{thm}\thmlabel{planartight}
There exists an absolute constant $c>0$ such that
for any $\Delta>1$ and $D\geq 10^6 \ln \Delta$ there exists a planar graph 
with diameter $\leq D$, maximum degree $\leq \Delta$,
and a vertex $s$ such that
with probability $1-o_n(1)$ we have $h(\tcal(G,s))>c\sqrt{\Delta}(D+\ln n)$.
\end{thm}

\begin{proof}

Let $C =10^5$, $a=e^2C$, 
$\delta=\Delta/2$, and 
$L=D\sqrt{\delta}/3a$, and 
Let $H$ be the graph shown in \figref{planar-h},
where each $V_i$ has  $\delta$ vertices.
Let $I$ be the perfect binary tree with $L$ leaves,
with each leaf-incident edge subdivided $aL/\sqrt \delta-1$ times.
Let $G$ be the graph obtained from identifying 
the $i$th leaf of $I$ with an arbitrary vertex from $V_i$.
Note that $G$ is a planar graph
with maximum degree $2\delta=\Delta$,
diameter $2 (aL/\sqrt{\delta}+1+\log_2 L)\leq D$,
and $n=\delta L + L - 1 + (2L-1) + L (aL/\delta-1)=O(\delta L + L^2/\delta)$
 vertices.
Let $s$ be an arbitrary vertex in $V_1$.
Since 
$L=\Omega(\sqrt{\Delta}(D+\ln n))$,
to complete the proof, we need only show that
with probability $1-o_n(1)$ we have 
 $h(\tcal(G,s)) \geq 2L-2$.

Choose an arbitrary vertex $t\in V_L$.
Let $\mathcal A$ denote the event
$\tau(H,s,t) \leq C L / \sqrt \delta$,
and let $\mathcal B$ denote the event
``for all pairs  $v$ and $w$ of leaves of $I$ we have
$\tau(I,v,w) > CL / \sqrt \delta$.
Note that if both $\mathcal A$ and $\mathcal B$ happen,
then the path in $\tcal(G,s)$ from $s$ to $t$ uses edges from $H$ only, which implies the height of this tree is at least $2L-2$.
To complete the proof via the union bound, we need only show that each of $\mathcal A$ and $\mathcal B$ happen with probability $1-o_L(1)=1-o_n(1)$.

We start with  $\mathcal A$.
  In $H$, one can go from the vertex in-between $V_i$ and $V_{i+1}$ to the vertex in-between
  $V_{i+1}$ and $V_{i+2}$ by taking a path whose weight is 
  distributed as a $Y_{\delta,1}$ random variable (recall the definition of a $Y_{a,b}$ random variable  from \secref{prelim}).
Therefore, we have 
\[
\tau(H,s,t) = X_1+X_2 + \sum_{i=1}^{L-2}Z_i \:,
\]
where $X_1,X_2$ are independent $\exponential(1)$
random variables (weights of the first and last edges),
and $Z_i$'s are independent $Y_{\delta,1}$ random variables.
Since $C/3\geq 3\times (64+1024)$,
Using \lemref{sqrtab} (concentration for the sum of $Y_{a,b}$ random variables) we have
\[
1-\Pr\{\mathcal A\}
\leq
2 \Pr\left\{X_1 >  CL / 3\sqrt \delta\right\}
+
\Pr\left\{ \sum_{i=1}^{L-1}Z_i >  CL / 3\sqrt \delta\right\}
\leq
2\exp\left(-CL / 3\sqrt \delta\right) +
\exp(-(L-2)/9)
=o_L(1)
\]

We now prove $\mathcal B$ happens with high probability.
The path connecting
any pair of leaves of $I$ contains at least $2aL/\sqrt \delta$ edges,
each of them having an independent exponential(1) weight.
Therefore, using union bound over all pairs and using~\eqref{head-bound} we get
\[
1-\Pr\{\mathcal B\}
\leq
\binom{L}{2} \times \Pr\{\erlang(2aL/\sqrt \delta,1)\leq CL/\sqrt \delta\}\leq
L^2 \times (eC/2a)^{2aL/\sqrt{\delta}} = o_L(1)\:,
\]
completing the proof.
\end{proof}

%

  \begin{figure}
    \begin{center}
      \includegraphics{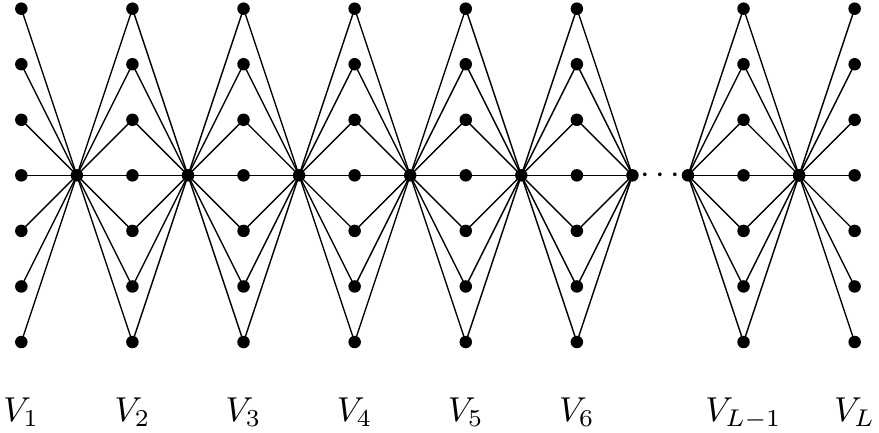}
    \end{center}
    \caption{The graph $H$ in the proof of \thmref{planartight}.}
    \figlabel{planar-h}
  \end{figure}
 
%

Next we describe a lower-bound construction that is $d$-degenerate,
has thickness $d$ and treewidth $O(d)$.  This construction shows that
\thmref{alantree-upper-bound-degenerate} is asymptotically tight for
all values $d\leq\Delta$ (with certain restrictions etc.).

\begin{thm}\thmlabel{alantree-lower-bound-degenerate}
There exists an absolute constant $c>0$ such that
for any $\Delta>1$ with
$D\geq 10^6 \ln \Delta$
and $d < \Delta$, 
 there exists a  graph  $G$
with diameter $\leq D$, maximum degree $\leq \Delta$, and the following properties:

(i) $G$ is $d$-degenerate, has thickness $\leq d$ and treewidth $\leq 2d+1$.

(ii) $G$ has a vertex $s$ such that
with probability $1-o_n(1)$ we have $h(\tcal(G,s))>c\sqrt{d\Delta}(D+\ln n)$.

\end{thm}

\begin{proof}
Let $C =10^5$, $a=e^2C$, 
$\delta=\Delta/2$, and 
$L=D\sqrt{d\Delta}/8a$, and 
Let $H$ be the graph shown in \figref{degenerate-h},
where each $V_i$ has $\delta$ vertices
  and each $V_i'$ has $d$ vertices,
  and each of the pairs
  $(V_1,V'_1)$, $(V'_1,V_2)$, $(V_2,V'_2)$, etc.\ form a complete bipartite graph.
Let $I$ be the perfect binary tree with $L$ leaves,
with each leaf-incident edge subdivided $aL/\sqrt{d\delta}-1$ times.
Consider the leaves of
$I$ in the order they are encountered in a depth first-traversal, for each $i\in\{1,\ldots,L\}$ identify the $i$th leaf of $I$
with some vertex in $V_i$. 
Let $G$ be the resulting graph.
Note that $G$ has maximum degree $2\delta=\Delta$,
diameter $\leq 2(1+aL/\sqrt{d\delta}+\log_2 L)\leq D$,
and $n=(\delta+D)L+2L-1+L(aL/d\delta-1)=O(\Delta L + L^2/d\Delta)$ vertices.
 
  \begin{figure}
    \begin{center}
       \includegraphics{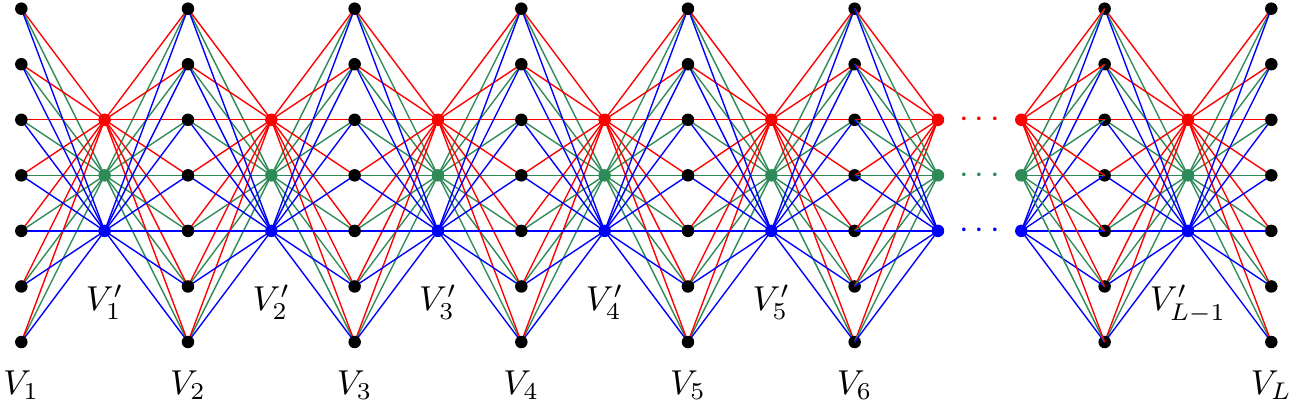}
    \end{center}
    \caption{The $d$-degenerate graph $H$ used in the proof of
       \thmref{alantree-lower-bound-degenerate}. In this example, $\delta=7$
        and $d=3$.}
    \figlabel{degenerate-h}
  \end{figure}
  
(i)
  Graph $G$ is $d$-degenerate because the vertices of degree greater
  than $d$ form an independent set. Therefore, every induced subgraph
  of $G$ is either an independent set (so has a vertex of degree 0)
  or contains a vertex of degree at most $d$.

  To see that $G$ has thickness $d$, 
  for each $i=1,\dots,L$, assign to each vertex of 
  $V_{i}'$ a distinct colour from one of $d$ colour classes.   Now partition the edges incident to
  these vertices among $d$ subgraphs depending on the color of the
  vertex they are incident to.  Edges not incident to these vertices
  can be assigned to any subgraph.  With this partition of edges, each
  subgraph becomes a subgraph of the planar graph used in the proof of
  \thmref{planartight}.

  To see that $G$ has treewidth $2d+1$, we build a tree decomposition
  of $G$ with bags of maximum size $2d+2$.  For convenience, we 
  define $V_0=V_{L+1}=\emptyset$.

  We begin with a tree $T'$ of empty bags that has the same shape as $I$. For each vertex $v$ of $I$, let $B_v$ denote the bag of $v$.
  \begin{enumerate}
    \item Assign each vertex of $v$ of $I$ to $B_v$ and to the (up to 2)
      children of $B_v$ in $T'$.

    \item Let $v_1,\ldots,v_L$ be the
     leaves of $I$ ordered so that each $v_i\in V_i$. In the leaf bag
     $B_{v_i}$ of $T'$ we add all vertices in
     $V_{i-1}'$ and $V'_{i}$.

     Now each vertex in $V_{i}'$ appears in $B_{v_i}$ and $B_{v_{i+1}}$;
     so we add all vertices of $V_i'$ to each of the bags on the path in $T'$ from $B_{v_i}$ 
     to $B_{v_{i+1}}$.

    \item  Finally, to each $B_{v_i}$ we attach $\delta-1$ bags as leaves
     of $T'$; in each bag we put all the vertices in $V_i' \cup V_{i+1}'$, and a distinct vertex of $V_i\setminus\{v_i\}$. We
     call each such bag $B_v$, where $v$ is the unique vertex of
     $V_i\setminus\{v_i\}$ contained in the bag.
  \end{enumerate}

  No bag contains more than $2d+2$ vertices: 
  for a leaf $v_i$,
  $B_{v_i}$ contains $v_i$ and its parent, as well as vertices in $V_{i-1}'\cup V_i'$.
  For a non-leaf vertex $v$ of $I$, observe that (in any binary tree)
  there are at most two distinct indices $i,j$ such that $v$ lies on the $(v_i,v_{i+1})$-path in $I$ and on
  the $(v_j,v_{j+1})$-path, hence $B_v$ contains $v$ and its parent, as well as possibly $V_i'$ and $V_j'$. For each $v\in V_i\setminus\{v_i\}$,
  $B_v$ contains at most $2d+1$ vertices; $v$ and the vertices in
  $V_{i-1}'\cup V_{i}'$.

  For each edge $vw$ of $G$, there is some bag that contains both $v$
  and $w$: If $vw$ is an edge of $T$ with $v$ a child of $w$ then $B_v$
  contains both $v$ and $w$. Otherwise, $v\in V_i$ and $w\in V_{i-1}'$
  or $w\in V_{i}'$, in which case $v$ and $w$ appear in $B_v$.

  Finally, for each vertex $v$ of $G$, the subgraph of $T'$ induced by
  bags containing $v$ is connected: For a vertex $v\in I$ this subgraph
  is either an edge or a single vertex.  For a vertex $v\in V_i$ this
  subgraph is a single vertex.  For a vertex $v\in V'_{i}$ this subgraph
  is a path joining two vertices of $T'$.

  Therefore, $T'$ is a tree-decomposition of $G$ whose largest bag
  has size $2d+2$, and thus treewidth of $G$ is at most $2d+1$.

(ii)
Let $s$ be an arbitrary vertex in $V_1$.
Since 
$L=\Omega(\sqrt{d\Delta}(D+\ln n))$,
to prove part (ii) we need only show that
with probability $1-o_n(1)$ we have 
 $h(\tcal(G,s)) \geq 2L-2$.

Choose an arbitrary vertex $t\in V_L$.
Let $\mathcal A$ denote the event
$\tau(H,s,t) \leq C L / \sqrt {d\delta}$,
and let $\mathcal B$ denote the event
``for all pairs  $v$ and $w$ of leaves of $I$ we have
$\tau(I,v,w) > CL / \sqrt {d\delta}$.
Note that if both $\mathcal A$ and $\mathcal B$ happen,
then the path in $\tcal(G,s)$ from $s$ to $t$ uses edges from $H$ only, which implies the height of this tree is at least $2L-2$.
To complete the proof via the union bound, we need only show that each of $\mathcal A$ and $\mathcal B$ happen with probability $1-o_L(1)=1-o_n(1)$.

We start with  $\mathcal A$.
  In $H$, one can go from a given vertex in $V'_i$ to some vertex in $V'_{i+1}$  by taking a path whose weight is 
  distributed as a $Y_{\delta,d}$ random variable.
Therefore, $\tau(H,s,t)$  is stochastically dominated by
\[
X_1+X_2 + \sum_{i=1}^{L-2}Z_i \:,
\]
where $X_1,X_2$ are independent $\exponential(1)$
random variables (weights of the first and last edges), and
$Z_i$'s are independent $Y_{\delta,d}$ random variables.
Since $C/3\geq 3\times(64+1024)$,
Using \lemref{sqrtab} (concentration for the sum of $Y_{a,b}$ random variables) we have
\[
1-\Pr\{\mathcal A\}
\leq
2 \Pr \{X_1 >  CL / 3\sqrt {d\delta}\}
+
\Pr\{ \sum_{i=1}^{L-2}Z_i >  CL / 3\sqrt {d\delta}\}
\leq
2\exp(-CL / 3\sqrt {d\delta}) +
\exp(-(L-2)/9)
=o_L(1)
\]

We now prove $\mathcal B$ happens with high probability.
The path connecting
any pair of leaves of $I$ contains at least $2aL/\sqrt{d \delta}$ edges,
each of them having an independent exponential(1) weight.
Therefore, using union bound over all pairs and using~\eqref{head-bound} we get
\[
1-\Pr\{\mathcal B\}
\leq
\binom{L}{2} \times \Pr\{\erlang(2aL/\sqrt{d \delta},1)\leq CL/\sqrt {d\delta}\}\leq
L^2 \times (eC/2a)^{2aL/\sqrt{d\delta}} = o_L(1)\:,
\]
completing the proof.
\end{proof}

\section*{Acknowledgements}

Some of this research took place at the Workshop on Random Geometric
Graphs and Their Applications to Complex Networks, at the Banff
International Research Station, November 6--11, 2016.  More of this
research took place at the Rio Workshop on Geometry and Graphs, at IMPA,
February 12--18, 2017.  In both cases, we are grateful to the workshop
organizers and other participants for providing stimulating working
environments.  We are especially grateful to G\'abor~Lugosi for helpful
discussions on many aspects of this work and to Tasos~Sidiropoulos
for asking us about planar graphs.
\bibliographystyle{plain}
\bibliography{alantree}

\begin{thebibliography}{10}

\bibitem{fpp_survey}
Antonio Auffinger, Michael Damron, and Jack Hanson.
\newblock 50 years of first passage percolation.
\newblock {\em arXiv}, 1511.03262~[math.PR], 2016.

\bibitem{balister.bollobas.ea:first-passage}
P.~N. Balister, B.~Bollob{\'a}s, A.~M. Frieze, and O.~M. Riordan.
\newblock The first-passage diameter of the cube, 2017.
\newblock Unpublished manuscript.

\bibitem{bollobas:isoperimetric}
B{\'{e}}la Bollob{\'{a}}s.
\newblock The isoperimetric number of random regular graphs.
\newblock {\em Eur. J. Comb.}, 9(3):241--244, 1988.

\bibitem{bollobas.kohayakawa:on}
B{\'e}la Bollob{\'a}s and Yoshiharu Kohayakawa.
\newblock On {R}ichardson's model on the hypercube.
\newblock In {\em Combinatorics, Geometry, and Probability (Cambridge 1993)},
  pages 129--137. Cambridge University Press, 1997.

\bibitem{boucheron2013concentration}
St{\'e}phane Boucheron, G{\'a}bor Lugosi, and Pascal Massart.
\newblock {\em Concentration inequalities: A nonasymptotic theory of
  independence}.
\newblock Oxford university press, 2013.

\bibitem{devroye:branching}
Luc Devroye.
\newblock Branching processes in the analysis of the heights of trees.
\newblock {\em Acta Informatica}, 24(3):277--298, 1987.

\bibitem{richardson_survey}
R.~Durrett.
\newblock Stochastic growth models: recent results and open problems.
\newblock In {\em Mathematical approaches to problems in resource management
  and epidemiology ({I}thaca, {NY}, 1987)}, volume~81 of {\em Lecture Notes in
  Biomath.}, pages 308--312. Springer, Berlin, 1989.

\bibitem{fill.pemantle:percolation}
James~Allen Fill and Robin Pemantle.
\newblock Percolation, first-passage percolation and covering times for
  {R}ichardson's model on the $n$-cube.
\newblock {\em The Annals of Applied Probability}, 3(2):593--629, 1993.

\bibitem{janson}
Svante Janson.
\newblock Tail bounds for sums of geometric and exponential variables.
\newblock available in \url{http://www2.math.uu.se/~svante/papers/sjN14.pdf}.

\bibitem{martinsson:unoriented}
Anders Martinsson.
\newblock Unoriented first-passage percolation on the $n$-cube.
\newblock {\em The Annals of Applied Probability}, 26(5):2597--2625, 2016.

\bibitem{McDiarmid1998}
Colin McDiarmid.
\newblock Concentration.
\newblock In Michel Habib, Colin McDiarmid, Jorge Ramirez-Alfonsin, and Bruce
  Reed, editors, {\em Probabilistic Methods for Algorithmic Discrete
  Mathematics}, pages 195--248. Springer Berlin Heidelberg, Berlin, Heidelberg,
  1998.
\newblock available at
  \url{http://www.stats.ox.ac.uk/people/academic_staff/colin_mcdiarmid/?a=4139}.

\bibitem{mohar.thomassen:graphs}
Bojan Mohar and Carsten Thomassen.
\newblock {\em Graphs on Surfaces}.
\newblock Johns Hopkins University Press, Baltimore, Maryland, 2001.

\bibitem{pittel:note}
Boris Pittel.
\newblock Note on the heights of random recursive trees and random m-ary search
  trees.
\newblock {\em Random Struct. Algorithms}, 5(2):337--348, 1994.

\bibitem{wolle.koster.ea:note}
Thomas Wolle, Arie~M.C.A. Koster, and Hans~L. Bodlaender.
\newblock A note on contraction degeneracy.
\newblock Technical Report UU-CS-2004-042, University of Utrecht, 2004.

\end{thebibliography}

\appendix

\section{Proof of~\lemref{sqrtab}}
\applabel{yab}

We will use the following inequality, which holds for any positive integer $k$ and any real number $\lambda$
(see \cite[Theorem~5.1(ii)]{janson}):
\begin{equation}
\eqlabel{simpleerlang}
\Pr\{\erlang(k,1) \geq \lambda k\}
\leq\exp(1-\lambda) \:.
\end{equation}

We will also use the following inequality, which holds for any binomial random variable $X$, and any $M\leq\E[X]$ (see \cite[Theorem~2.3(c)]{McDiarmid1998}):
\begin{equation}
\eqlabel{binomialhalf}
\Pr\{X < M/2\} \leq \exp(-M/8)\:.
\end{equation}

We will use the following version of
Bernstein's inequality (see Theorem~2.10 and Corollary~2.11
in~\cite{boucheron2013concentration}).


\begin{thm}[Bernstein's inequality]
\label{thm:bernstein}
Let $X_1,\dots,X_m$ be non-negative independent random variables for which there exist $v,c$ satisfying
\[\sum_{i=1}^m \E [X_i^p] \leq v p!  c^{p-2}/2\]
for all positive integers $p\geq2$.
Then for any $t>0$ we have
\[
\Pr \left\{ \sum_{i=1}^m (X_i - \E[X_i]) \geq ct + \sqrt{2vt}\right\}\leq e^{-t} \:,
\]
and
\[
\Pr \left\{ \sum_{i=1}^m (X_i - \E[X_i]) \geq t
\right\}\leq \exp\left(-\frac{t^2}{2v+2ct}\right) \:,
\]
\end{thm}

We begin with a helper lemma.

\begin{lem}
\label{lem:yabconc}
For any $t$ we have
\[
\Pr\{Y_{a,b}>t\} \leq \exp(-at/64)+\exp(-abt^2/1024)\:.
\]
\end{lem}
\begin{proof}
First, consider the case $t>4$.
Note that there exist $a$ independent root-to-leaf paths, the weight of each is $\erlang(2,1)$.
Hence, using~\eqref{simpleerlang} and since $t\geq4$,
\[
\Pr\{Y_{a,b}>t\}
\leq
\Pr\{\erlang(2,1)>t\}^a
\leq
(\exp(1-t/2))^a
\leq
(\exp(-t/4))^a
=\exp(-at/4)
\leq\exp(-at/64)\:.
\]

The case $t\leq 0$ is trivial, so we consider the case $0\leq t\leq 4$.
Note that for such $t$ we have $1-\exp(-t/2)\geq t/8$.
We say a node in the tree \emph{survives} if each of the edges on its path to the root
have weight at most $t/2$.
Note that $Y_{a,b}>t$ implies no node at level 2 survives.
The probability that a node at level 1 (children of the root) survives is $1-\exp(-t/2)$,
so the number of surviving nodes at level 1, $S_1$, is a binomial random variable with mean $a(1-\exp(-t/2))\geq at/8$.
From~\eqref{binomialhalf} we have
\[
\Pr\{S_1 < at/16\}
\leq
\Pr\{S_1 < \E[S_1] / 2\}
\leq
\exp(-\E[S_1]/8)
\leq 
\exp(-at/64)\:.
\]
Conditioned on $S_1\geq at/16$, the number of surviving nodes at level 2, $S_2$, is a binomial random variable with mean
$S_1b(1-\exp(-t/2))\geq abt^2/128$, so using\eqref{binomialhalf} again we have
\begin{align*}
\Pr\{Y_{a,b}>t | S_1\geq at/16\}
& \leq
\Pr\{S_2=0 | S_1\geq at/16\}
\leq
\Pr\{S_2<abt^2/256 | S_1\geq at/16\} \\
& \leq
\exp(-abt^2/1024)\:,
\end{align*}
completing the proof.
\end{proof}

We are now ready to prove
\lemref{sqrtab}.
Let $X_1,\dots,X_m$ be i.i.d.\ distributed as $Y_{a,b}$ for some $a,b$.
Then we want to prove
$\E[X_1] =O(1/a+1/\sqrt{ab})$ and moreover,
\[
\Pr\left\{\sum_{i=1}^m X_i \geq 3 m (64/a + 1024/\sqrt{ab}) \right\}
\leq
\exp(-m/9)\:.
\]

Let $d_1 = a/64$ and $d_2=ab/1024$.
For any positive integer $p$, by Lemma~\ref{lem:yabconc} we have
\[
\E[X_1^p] = \int_0^{\infty} \Pr \{X_1>t^{1/p}\}dt
\leq
\int_0^{\infty} \exp(-d_1 t^{1/p}) + \int_0^{\infty} \exp(-d_2 t^{2/p})
\]
For any positive numbers $c,\alpha$, we have
\begin{align}
\int_0^{\infty} \exp(-c t^{\alpha}) dt
& = 
\int_0^{\infty} \exp(-x) \frac {x^{1/\alpha-1}}{\alpha c^{1/\alpha}}dx
=
\frac{c^{-1/\alpha}}{\alpha}\int_0^{\infty}
e^{-x}x^{1/\alpha-1}
= \frac{c^{-1/\alpha}\Gamma(1/\alpha)}{\alpha},
\eqlabel{expintegral}
\end{align}
whence,
\[
\E[X_1^p] \leq pd_1^{-p}\Gamma(p)
+ pd_2^{-p/2}\Gamma(p/2)/2
\]
In particular, setting $p=1$ gives $\E[X_1]\leq 64/a+1024/\sqrt{ab}\eqqcolon c$.
Let $v=4c^2m$.
For $p\geq 2$, we have
\begin{align*}
\sum_{i=1}^m \E [X_i^p] 
\leq
mpd_1^{-p}\Gamma(p)
+ mpd_2^{-p/2}\Gamma(p/2)/2
\leq
m p! d_1^{-p}
+ m p! d_2^{-p/2}/2
\leq v p! c^{p-2}/2\:.
\end{align*}
Bernstein's inequality (Theorem~\ref{thm:bernstein}) gives that for all $t$,
\[
\Pr\left\{\sum_{i=1}^m X_i \geq 
m \left(64/a+1024/\sqrt{ab}\right)
+ ct + 3c\sqrt {mt} \right\} \leq e^{-t}\:,
\]
and choosing $t=m/9$ completes the proof of the lemma.

\pagebreak
\newpage

\end{document}